\pdfoutput=1
\RequirePackage{ifpdf}
\ifpdf 
\documentclass[pdftex]{sigma}
\else
\documentclass{sigma}
\fi

\numberwithin{equation}{section}
\newtheorem{Theorem}{Theorem}[section]
\newtheorem*{Theorem*}{Theorem}
\newtheorem{Corollary}[Theorem]{Corollary}
\newtheorem{Lemma}[Theorem]{Lemma}
\newtheorem{Proposition}[Theorem]{Proposition}
 { \theoremstyle{definition}

\newtheorem{Example}[Theorem]{Example}
\newtheorem{Remark}[Theorem]{Remark} }

\begin{document}
\allowdisplaybreaks

\renewcommand{\thefootnote}{}

\newcommand{\arXivNumber}{2212.07915}

\renewcommand{\PaperNumber}{028}

\FirstPageHeading

\ShortArticleName{CYT and SKT Metrics on Compact Semi-Simple Lie Groups}

\ArticleName{CYT and SKT Metrics on Compact Semi-Simple\\ Lie Groups\footnote{This paper is a~contribution to the Special Issue on Differential Geometry Inspired by Mathematical Physics in honor of Jean-Pierre Bourguignon for his 75th birthday. The~full collection is available at \href{https://www.emis.de/journals/SIGMA/Bourguignon.html}{https://www.emis.de/journals/SIGMA/Bourguignon.html}}}

\Author{Anna FINO~$^{\rm ab}$ and Gueo GRANTCHAROV~$^{\rm b}$}

\AuthorNameForHeading{A.~Fino and G.~Grantcharov}

\Address{$^{\rm a)}$~Dipartimento di Matematica ``G. Peano'', Universit\`{a} degli studi di Torino,\\
\hphantom{$^{\rm a)}$}~Via Carlo Alberto 10, 10123 Torino, Italy}
\EmailD{\href{mailto:annamaria.fino@unito.it}{annamaria.fino@unito.it}}

\Address{$^{\rm b)}$~Department of Mathematics and Statistics, Florida International University,\\
\hphantom{$^{\rm b)}$}~Miami, FL 33199, USA}
\EmailD{\href{mailto:afino@fiu.edu}{afino@fiu.edu}, \href{mailto:grantchg@fiu.edu}{grantchg@fiu.edu}}

\ArticleDates{Received January 02, 2023, in final form May 11, 2023; Published online May 25, 2023}

\Abstract{A Hermitian metric on a complex manifold $(M, I)$ of complex dimension $n$ is called Calabi--Yau with torsion (CYT) or Bismut--Ricci flat, if the restricted holonomy of the associated Bismut connection is contained in ${\rm SU}(n)$ and it is called strong K\"ahler with torsion (SKT) or pluriclosed if the associated fundamental form $F$ is $\partial \overline \partial$-closed. In the paper we study the existence of left-invariant SKT and CYT metrics on compact semi-simple Lie groups endowed with a Samelson complex structure $I$. In particular, we show that if $I$ is determined by some maximal torus $T$ and $g$ is a left-invariant Hermitian metric, which is also invariant under the right action of the torus $T$, and is both CYT and SKT, then $g$ has to be Bismut flat.}


\Classification{53C55; 53C05; 22E25; 53C30; 53C44}

\begin{flushright}
\begin{minipage}{65mm}
\it Dedicated to Jean-Pierre Bourguignon\\ for his 75th birthday
\end{minipage}
\end{flushright}

\renewcommand{\thefootnote}{\arabic{footnote}}
\setcounter{footnote}{0}

\section{Introduction}

On every Hermitian manifold $(M, I, g)$ there exists a one-parameter family of Hermitian connections which can be distinguished by their torsion tensor and coincide with the Levi-Civita connection when $M$ is K\"ahler. Among them there are
the Chern connection $\nabla^{\rm C}$ on the holomorphic tangent bundle and the Bismut (or Strominger) connection $\nabla^{\rm B}$, which is also called by physicists the K\"ahler with torsion (KT) connection. $\nabla^{\rm B}$ is the unique Hermitian connection whose torsion tensor is totally skew-symmetric and its torsion $T^{\rm B}$ is characterized by the condition
\[
g \big(T^{\rm B}(X,Y), Z\big) = {\rm d}F(IX,IY,IZ),
\]
where $F$ is the fundamental form $F(X,Y) = g(IX,Y)$, and $X$, $Y$, $Z$ are smooth vector fields. Since $\nabla^{\rm B}$ is a Hermitian connection, its holonomy is contained in the unitary group~${\rm U}(n)$, where~$n$ is the complex dimension of $M$. If the holonomy of $\nabla^{\rm B}$ is reduced to ${\rm SU}(n)$, the metric~$g$
is said to be Calabi--Yau with torsion (shortly CYT). It is also called {\it Bismut Ricci-flat}.
This type of geometry in the physical context was considered first by Strominger~\cite{Strominger} and
Hull~\cite{Hull}. In~\cite{GIP}, it was conjectured that every compact complex manifold with vanishing first Chern class admits a~CYT metric. Counterexamples to this conjecture appear in~\cite{FG}. In~\cite{Grantcharov} it has been shown that every compact complex homogeneous space with vanishing first Chern class, after an appropriate deformation of the complex structure, admits a homogeneous CYT metric, provided that the complex homogeneous space also has an invariant volume form.

{\em SKT} metrics, which are called also {\em pluriclosed}, are Hermitian metrics whose fundamental form $F$ is $\partial \overline \partial$-closed or equivalently, whose torsion $3$-form of the Bismut connection is closed. Interest in SKT manifolds stems from various sources~\cite{Swann}. They occur in physics in the context of supersymmetric theories, see for example \cite{GHR,HP,Strominger}. Secondly, every conformal class of a~hermitian metric on a~compact complex surface contains an SKT metric \cite{Gau}, but this property is no longer true in higher dimensions. Recently the SKT and CYT metrics have been studied in relation to the pluriclosed flow and the generalized Ricci flow \cite{AFSU, Barbaro, GJS, GS}.

Basic examples of SKT manifolds are given by compact Lie groups endowed with a bi-invariant metric and any of the compatible Samelson’s complex structures \cite{Samelson}. These SKT metrics are also CYT, but the associated Bismut connection is flat. Note that a simply connected compact Riemannian manifold admitting a flat metric connection with closed skew symmetric torsion is isometric to a product of compact simple Lie groups with bi-invariant metrics (see~\cite[Theorem~3.54]{GS}),
however it is currently not known whether there are other metrics, which are both SKT and CYT, but with non-flat Bismut connection. Many of the CYT examples in \cite{Grantcharov} have non-vanishing curvature of the Bismut connection. Recently in~\cite{PR, PR2} non-Hermitian examples which are Bismut Ricci-flat, which have closed torsion form and non-vanishing Bismut curvature have been constructed on compact homogeneous spaces.

By \cite{Samelson}, every compact Lie group G of even dimension admits a complex structure such that left translations are holomorphic mappings and the complex structure is determined by a~choice of a maximal torus $T$. This result is an extension of Borel's theorem which states that the quotient of a compact Lie group by its maximal torus always has a homogeneous complex structure.
Note that the Samelson’s construction depends on a choice of the maximal torus and such complex structures are compatible if they are compatible with the metric restricted to this torus.

More generally, the existence of SKT metrics on the class of Wang C-spaces, which are defined as compact complex manifolds admitting a transitive action by a compact Lie group of biholomorphisms and finite fundamental group, was investigated in \cite{FGV}. By \cite{FGV}, it turns out that SKT structures could only appear on a product of a compact Lie group and a K\"ahler homogeneous C-space (generalized flag manifold). Note that these spaces admit left-invariant SKT metrics.

Therefore, in view of the above, it is natural to investigate whether on an even dimensional compact Lie group with a Samelson complex structure there exist left-invariant metrics which are both SKT and CYT, but not Bismut flat.

Motivated by this question, in the current paper we study left-invariant SKT and CYT metrics on compact semi-simple Lie groups endowed with a Samelson's complex structure $I$. Using a~symmetrization, we show that the existence of left-invariant SKT or CYT metrics leads to an existence of such metrics which are left-invariant and right $T$-invariant. As a first consequence we show that if $g$ is a left-invariant SKT metric compatible with $I$, then the complex structure~$I$ has to be compatible with a bi-invariant metric $g_0$. Then we provide characterizations of left-invariant SKT and CYT metrics which are also invariant under the right $T$-action. As a consequence, we prove that on every compact semi-simple Lie group a metric of this type which is both CYT and SKT must be a bi-invariant one, or equivalently a product of metrics proportional to the Killing forms on the simple factors. Therefore, in particular, such metric must be Bismut flat. We also construct an explicit $5$-parameter family of SKT metrics on ${\rm SO}(9)$, which are $T$-invariant for the natural choice of the torus $T$.

\section{Preliminaries on compact semi-simple Lie groups}
It was shown independently by Samelson \cite{Samelson} and Wang \cite{Wang} that every compact Lie group $G$ of even dimension admits left-invariant complex structures $I$.
Moreover, there exists a maximal toral subgroup $T$ of $G$ that is $I$-complex and the quotient $G/T$
is a complete flag variety of $G$ with a complex structure such that the quotient map $\pi\colon G \rightarrow G/T$ is holomorphic.

We now review the Samelson's construction of left-invariant complex structures on compact semi-simple Lie groups.
We will follow the conventions of Alekseevsky--Perelomov \cite{AlP}, and Bordemann--Forger--Romer \cite{BFR}.
 Let $G^{\rm c}$ be the complexification of a compact semi-simple Lie group $G$. Then $G$ may be regarded as a Lie subgroup of $G^{\rm c}$, and the Lie algebra of $G^{\rm c}$ is the complexification $\mathfrak g^{\rm c}$ of the Lie algebra $\mathfrak g$ of $G$.

If we fix a maximal abelian subalgebra $\mathfrak t$ of $\mathfrak g$, then $\mathfrak h= \mathfrak t^{\rm c}$ is a maximal abelian subalgebra of~$\mathfrak g^{\rm c}$ and we have the root space decomposition
\[
\mathfrak g^{\rm c} = \mathfrak{h} \oplus \bigoplus_{\alpha \in \Phi} \mathfrak g_{\alpha},
\]
where $\Phi$ is the finite subset of nonzero elements of ${\mathfrak h}^*$ called roots, and each $\mathfrak g_{\alpha}$ is the complex $1$-dimensional subspace of $\mathfrak g^{\rm c} $ given by
\[
\mathfrak g_{\alpha} := \{ X \in {\mathfrak g}^{\rm c} \mid [H , X] = \alpha (H) X, \, \forall H \in \mathfrak{h} \}.
\]

Therefore, there exists a system of positive roots, that is a set $P \subset \Phi$ satisfying $P \cap (- P) = \varnothing$, $ P \cup (-P)=\Phi$, and $\alpha, \beta \in P$, $ \alpha + \beta \in \Phi$ $\Rightarrow$ $\alpha + \beta \in P$, and such that ${\rm {span}}_{\mathbb C}(P) = \mathfrak h$. We note that $\Phi$ is also invariant under a set of reflections which leads to the classification of all irreducible root systems. Moreover, the structure of $\Phi$ determines the structure of $\mathfrak g^{\rm c}$ in a sense that $[\mathfrak g_{\alpha}, \mathfrak g_{\beta}] \subseteq \mathfrak g_{\alpha +\beta}$ when $\alpha+\beta \in \Phi$, $[\mathfrak h,\mathfrak g_{\alpha}] = \mathfrak g_{\alpha}$ and $[\mathfrak g_{\alpha}, \mathfrak g_{-\alpha}] \subseteq \mathfrak h$ and the other commutators vanish.

Among the roots in $P$ there is a set of simple roots, which we denote by $\alpha_1, \alpha_2,\dots ,\alpha_r$ and forms a basis of the space $\mathfrak t^*$. For any $\alpha \in P$, we have $\alpha = \sum_{i =1}^r m_i\alpha_i$, with $m_i$ being a non-negative integer. The basis $(\alpha_l)$ also determines a partial order of the set of positive roots and~$r$ is called rank of $\mathfrak g^{\rm c}$.

Denote by $B(X,Y) = {\rm tr} ({\rm ad}_X \circ {\rm ad}_Y)$ the Killing--Cartan bilinear form of $\mathfrak g^{\rm c}$. The restriction of $B$ to $\mathfrak t$ is non-degenerate and $\mathfrak g_{\alpha}$ is $B$-orthogonal to $\mathfrak g_{\beta}$ for $\beta \neq - \alpha$.
 For every $\alpha \in P$ we choose $(E_{\pm \alpha}, H_{\alpha})$ such that $\mathfrak g_{\pm \alpha}= {\rm {span}}_{\mathbb C} \langle E_{\pm\alpha}\rangle$, $[E_{\alpha}, E_{-\alpha}]=H_{\alpha}$, $B(E_{\alpha}, E_{-\alpha}) = 1$, $[H, E_{\alpha}] = \alpha(H)E_{\alpha}$ and $B(H, H_{\alpha}) = \alpha(H)$ for all $H \in \mathfrak t$.

Note that $\mathfrak t = {\rm span}_{\mathbb R} \langle {\rm i} H_{\alpha}\rangle$ is the Lie algebra of $T$. Moreover, since all compact forms of a~complex semi-simple Lie algebra are conjugate, we can assume that $\mathfrak g$ is the real span of $X_{\alpha}$,~$Y_{\alpha}$,~${\rm i}H_{\alpha}$, where $X_{\alpha} = \frac{1}{\sqrt{2}}(E_{\alpha}+E_{-\alpha}), Y_{\alpha} = \frac{1}{\sqrt{2}{\rm i}}(E_{\alpha}-E_{-\alpha}).$
In particular, $\{ X_{\alpha}, Y_{\alpha} \} $ is an orthonormal set with respect to $- B$ on $G$.

When the rank $r$ is even, Samelson proved that a choice of $P$ and of a complex structure on~$\mathfrak t $ determine a left-invariant complex structure $I$ on $G$, which is given by $I(E_{\alpha})= {\rm i}E_{\alpha}$, $I(E_{-\alpha})=-{\rm i}E_{-\alpha}$, for every $\alpha \in P$, and by a linear endomorphism of $\mathfrak{t}$ with square $-{{\rm Id}}$.

Note that if $ \mathfrak g = \sum_{i =1}^m \mathfrak g_i$, where $\mathfrak g_i$ are simple, then $\mathfrak t= \sum_{i=1}^m \mathfrak t_i$ and $B = \sum_{i=1}^m B_i$. We denote by $E_{\alpha}^{(i)}$, $E_{-\alpha}^{(i)}$ the corresponding root vectors in $\mathfrak g_i$ with respect to $\mathfrak t_i$ and by $P_i$ the set of positive roots of the simple factor $G_i$ with respect to $\mathfrak t_i$. We will not use the upper index when there is no confusion.

\begin{Remark} \label{remarkAkekseevski} Let $I$ be a Samelson complex structure on $G$. If $h$ is a left-invariant $I$-Hermitian metric on $G$ which is also invariant by the right $T$-action on $G$, then there exists a metric on $G/T$ such that $\pi\colon G\rightarrow G/T$ is a Riemannian submersion. Since the spaces $\langle X_{\alpha}, Y_{\alpha}\rangle$ are $T$-invariant and irreducible,
the metric $h$ can be written, up to a constant, as $h(X,Y)= -B (\Lambda(X),Y)$, where~$\Lambda$ is
a symmetric operator defined by $\Lambda(E_{\alpha})=\lambda_{\alpha} E_{\alpha}$ for every $\alpha \in P$ and any
appropriate matrix on $ \mathfrak t$ such that $h$ is positive definite (see, e.g., \cite{Arv}).
\end{Remark}

\section{Invariant SKT metrics}

 We can apply the symmetrization with respect to the right $T$-action to prove that if we have a~left-invariant SKT $I$-Hermitian metric,
then it is possible to construct a SKT metric which is right $T$-invariant.

\begin{Lemma}\label{llemma1}
Let $G$ be a compact semi-simple Lie group of even rank endowed with a Samelson complex structure $I$. Let $g$ be a left-invariant SKT $I$-Hermitian metric.
Then symmetrizing~$g$ by the right $T$-action on $G$, we obtain a Hermitian metric $h$ which is still SKT and also left-invariant and right $T$-invariant.
\end{Lemma}

\begin{proof} Since $I$ is invariant under the right $T$-action by the aforementioned result of Wang \cite{Wang}, we can symmetrize $g$. Indeed, let $d \nu$ be the standard volume on the torus $T$ and denote by $R_h$ the right translation by $h \in T$ on $G$. We can define
\[
 \overline{g}_x ({\bf u}, {\bf v}) := \int_T (R_h^* g)_x ({\bf u}, {\bf v}) \, {\rm d} \nu (h),
\]
for every $x \in G, {\bf u}, {\bf v} \in T_x G$. By construction $\overline{g}$ is invariant by the right $T$-action on $G$ and compatible with $I$. Moreover, since left and right translations commute and right $T$- translations are holomorphic, $\overline{g}$ is left-invariant.

Since the SKT condition ${\rm d}{\rm d}^{\rm c} \omega= 0$ is linear in $\omega$ and ${\rm d}$, ${\rm d}^{\rm c}$ commute with $R_h^*$, for every $ h \in T$, $\overline{g}$ is SKT. \end{proof}

Given a Samelson complex structure $I$, we can describe all possible left-invariant SKT $I$-Hermitian metrics which are also invariant by the right $T$-action on $G$.

\begin{Theorem}\label{lemma2}
Let $G$ be a compact semi-simple Lie group of even rank endowed with a Samelson complex structure $I$ and let $g$ be a left-invariant SKT $I$-Hermitian metric.
Then the restriction of $g$ on the fibers of $\pi\colon G \rightarrow G/T$ coincides with the restriction of a bi-invariant metric on $G$. As a consequence, the complex structure $I$ has to be compatible with a bi-invariant metric $g_0$.
\end{Theorem}

\begin{proof} By Lemma \ref{llemma1}, using the symmetrization it is not restrictive to suppose that $g$ is invariant under the right $T$-action.
 Suppose that $G$ up to a finite cover is given by the product $G_1\times G_2\times \dots \times G_m$ of simple Lie groups $G_i$ ($ \dim G_i > 1$). Then $\mathfrak t = \sum_{i = 1}^m \mathfrak t_i$, with $\mathfrak t_i$ the Cartan subalgebra of the Lie algebra $\mathfrak g_i$ of every factor $G_i$.

Select a $g$-unitary basis $(\theta_1, \dots , \theta_{2k})$ of $\mathfrak t^*$, such that $ I(\theta_{2l-1})=\theta_{2l}$.
Then by \cite{GGP}, we have $F = \sum_{l = 1}^k \theta_{2l-1}\wedge\theta_{2l} + \pi^*(F_b),$ where $F_b$ is a $(1,1)$-form of an invariant Hermitian structure on~$G/T$. Also by \cite[Section~5]{GGP},
\[
{\rm d}{\rm d}^{\rm c} F = \pi^*{\rm d}{\rm d}^{\rm c} F_b - \sum_{l = 1}^k \big(\omega_{2l-1}^2+\omega_{2l}^2\big),
\]
 where $\omega_i = {\rm d}\theta_i$ are the curvature forms of the connection on $\pi\colon G\rightarrow G/T$ determined by $\theta_i$. This in particular means that
 \[
 \sum_{l=1}^k \big([\omega_{2l-1}]^2+[\omega_{2l}]^2\big) = \sum_{i=1}^{2k} [{\rm d}\theta_i]^2 = 0
 \]
 in the cohomology $H^2(G/T, \mathbb{R})$.

Now, the flag manifold $G/T$ decomposes as $G/T = G_1/T_1\times\dots \times G_m/T_m$ and by the results of Chevalley \cite{Chevalley}, $H^2(G_i/T_i, \mathbb{R})$ is generated by the (exterior derivatives of the) simple roots in~$\mathfrak t_i^*$, and moreover, there is a unique quadratic relation among them. This quadratic relation can be written as $\sum_{j=1}^{s_i} \big[{\rm d}\widetilde{\theta}^{(i)}_j\big]^2 = 0$ for
a $-B_i$-orthonormal basis $\big(\widetilde{\theta}^{(i)}_1, \dots, \widetilde{\theta}^{(i)}_{s_i}\big)$ of $\mathfrak t_i^*$, which can be easily checked by direct calculations. Since $B =\sum_{i =1}^m B_i$, the basis is also $-B$-orthonormal.

It follows that the only quadratic relations among the cohomology classes in $H^2(G/T, \mathbb{R})$ are given by
\begin{equation}\label{quadraticrel}
 \sum_{i=1}^m c_i\sum_{j=1}^{s_i} \big[{\rm d}\widetilde{\theta}^{(i)}_j\big]^2 = 0, \end{equation} for some constants $c_1, \dots, c_m$.

Then
\begin{equation} \label{quadrelation}
\sum_{i=1}^{2k} [{\rm d}\theta_i]^2 = \sum_{i=1}^m c_i\sum_{j=1}^{s_i} \big[{\rm d}\widetilde{\theta}^{(i)}_j\big]^2.
\end{equation}

Denote the $-B$-orthonormal basis $\big(\widetilde{\theta}_1^{(1)}, \dots, \widetilde{\theta}_{s_1}^{(1)}, \dots, \widetilde{\theta}_1^{(m)}, \dots, \widetilde{\theta}_{s_m}^{(m)}\big)$ by $\sigma_l$ for $l=1,2,\dots ,2k$. Then $\theta_i = \sum_{l=1}^{2k} q_{li}\sigma_l$ for a transition matrix $Q = (q_{ij})$
and we can rewrite $\sum_{i=1}^{2k} [{\rm d}\theta_i]^2$ as
\[
\sum_{i=1}^{2k} [{\rm d}\theta_i]^2 = \sum_{i = 1}^{2k} \sum_{l,j =1}^{2k} q_{li} q_{ji} [{\rm d}\sigma_l \wedge {\rm d} \sigma_j].
\]
 On the other hand,
\[
 \sum_{i=1}^m c_i\sum_{j=1}^{s_i} \big[{\rm d}\widetilde{\theta}^{(i)}_j\big]^2 = c_1 \sum_{i =1}^{s_1} [{\rm d} \sigma_i]^2 + c_2 \sum_{i =1}^{s_2} [{\rm d} \sigma_{s_1 +i}]^2 + \dots + c_m \sum_{i =1}^{s_m} [{\rm d} \sigma_{s_1 + \dots + s_{m-1} + i} ]^2.
\]
Therefore, from the equality \eqref{quadrelation} we have
\[
 \sum_{i = 1}^{2k} \sum_{l,j =1}^{2k} q_{li} q_{ji} [{\rm d} \sigma_l \wedge {\rm d} \sigma_j] = c_1 \sum_{i =1}^{s_1} [{\rm d} \sigma_i]^2 + c_2 \sum_{i =1}^{s_2} [{\rm d} \sigma_{s_1 +i}]^2 + \dots + c_m \sum_{i =1}^{s_m} [{\rm d} \sigma_{s_1 + \dots + s_{m-1} + i} ]^2.
\]
Using that $[{\rm d}\sigma_l\wedge {\rm d}\sigma_j]$ are independent, for every $l\neq j$ we get that
\[
\sum_{i = 1}^{2k} q_{li} q_{ji} =0, \qquad l \neq j,
\]
and by \eqref{quadraticrel}
\[
\sum_{i =1}^{s_1} (q_{li})^2 = c_1, \quad \dots, \quad \sum_{i =s_{m-1}}^{s_m} (q_{li})^2 = c_m,
\]
i.e., that $QQ^t= D$, where $D$ is a diagonal matrix with entries $c_i$. From here we conclude, that~$c_i > 0$ and $Q$ preserves the metric $-\sum_{i = 1}^m c_i B_i$. In particular $\theta_i$ are $-\sum_{i=1}^m c_i B_i$-or\-thonor\-mal. As the $\theta_i$s are $g$- and $-\sum_{i=1}^m c_i B_i$-orthonormal, the restriction of $g$ to $\mathfrak t$ coincides with the restriction of a biinvariant metric, the theorem follows.
\end{proof}

\begin{Remark} It is well known that on an even dimensional compact Lie group equipped with a~bi-invariant metric $g_0$ there is always a $g_0$-compatible complex structure, which is necessary SKT and also Bismut-flat. But there exist Samelson's complex structures which are not compatible with any bi-invariant metric on $G$ as long as the rank of $G$ is at least 2.
\end{Remark}

As a consequence of Theorem \ref{lemma2}, we have the following

\begin{Corollary} \label{th4.1bis}
Let $G$ be a compact semi-simple Lie group with Samelson's complex structure $I$ as above. Let $g$ be a left-invariant $I$-Hermitian metric on $G$, which is also invariant under the right action of the torus $T$ in the Tits fibration $G\rightarrow G/T$, where $T$ is a maximal torus in $G$. We can write the fundamental form $F$ of the metric $g$
as $F = F|_{\mathfrak t} + \sum_{i=1}^m \sum_{\alpha \in P_i} \lambda_{\alpha} ^{(i)} e_{\alpha}^{(i)} \wedge I e_{\alpha}^{(i)}$, where $e_{\alpha}^{(i)}$, $Ie_{\alpha}^{(i)}$ are dual to $X_{\alpha}^{(i)}, Y_{\alpha}^{(i)}$ and $\lambda_{\alpha} ^{(i)}= g (E_{\alpha}^{(i)}, E_{-\alpha}^{(i)})$. Then
 $g$ is SKT iff $g |_{\mathfrak t} = - \sum_{i =1}^m \nu_i B|_{\mathfrak t_i}$ and $ \sum_{i =1}^m \sum_{\alpha \in P_i} \big(\lambda_{\alpha}^{(i)} - \nu_i\big) e_{\alpha}^{(i)} \wedge I e_{\alpha}^{(i)}$ is a linear combination of the basis $\omega_j = {\rm d} \alpha_j$ of $H^2(G/T, \mathbb R)$, where $\alpha_j$ are the simple roots.

Equivalently, $g$ is SKT iff ${\rm d}\big( \sum_{i =1}^m \sum_{\alpha \in P_i} \big(\lambda_{\alpha}^{(i)} - \nu_i\big) e_{\alpha}^{(i)}\wedge I e_{\alpha}^{(i)}\big)=0$ and $\Lambda |_{\mathfrak t} = \sum_{i=1}^m \nu_i {\rm Id}|_{\mathfrak t_i}$, for some constants $\nu_i$, $i= 1,2, \dots, m$.
\end{Corollary}

\begin{proof} By \cite{GGP}, we can rewrite $F = \sum_{l = 1}^k \mu_l^2 \theta_{2l-1}\wedge\theta_{2l} + \pi^*(F_b),$
where $F_b$ is a (1,1)-form of an invariant Hermitian structure on $G/T$ and $(\theta_i)$ is a $(-B)$-orthonormal basis of $\mathfrak t^*= \sum_{i =1}^m \mathfrak t_i^*$, which is also $g$-orthogonal. Note that $\theta_i$ are linear combinations of the simple roots $\alpha_i$ and
\[
F |_{\mathfrak t} = \sum_{l = 1}^k \mu_l^2 \theta_{2l-1}\wedge\theta_{2l}, \qquad \pi^*(F_b) = \sum_{i =1}^m \sum_{\alpha \in P_i} \lambda_{\alpha}^{(i)} e_{\alpha}^{(i)} \wedge I e_{\alpha}^{(i)},
\]
where $P_i$ denotes the set of positive roots of the simple factor $G_i$ with respect to $\mathfrak t_i$.
Suppose that $g |_{\mathfrak t} = - \sum_{i=1}^m \nu_i B|_{\mathfrak t_i}$ and $ \sum_{i =1}^m \sum_{\alpha \in P_i} \big(\lambda_{\alpha}^{(i)} - \nu_i\big) e_{\alpha}^{(i)} \wedge I e_{\alpha}^{(i)}$ is a linear combination of the basis $\omega_j = {\rm d} \alpha_j$ of $H^2(G/T, \mathbb R)$. Using that every $I$-Hermitian bi-invariant metric is SKT, we have that ${\rm d}{\rm d}^{\rm c} \big(F |_{\mathfrak t} + \nu_i e_{\alpha}^{(i)} \wedge I e_{\alpha}^{(i)}\big) =0$.
Since ${\rm d}\big(\sum_{i =1}^m \sum_{\alpha \in P_i } \big(\lambda_{\alpha}^{(i)} - \nu_i\big) e_{\alpha}^{(i)} \wedge I e_{\alpha}^{(i)}\big)=0$, then it follows that $g$ is SKT.

Conversely, if $g$ is SKT, then by Theorem~\ref{lemma2}, we have that $I$ is compatible with $- \sum_{i =1}^m \nu_i B_i$ and $g |_{\mathfrak t} = - \sum_{i =1}^m \nu_i B_i |_{\mathfrak t_i}$, for some positive $\nu_i$. This implies that $\Lambda$ is of the form $ \sum_{i=1}^m \nu_i {\rm Id}|_{\mathfrak t_i}$. Then, the difference between $F$ and the fundamental form associated to the metric $- \sum_{i =1}^m \nu_i B_i$ is given by $ \sum_{i =1}^m \sum_{\alpha \in P_i } \big(\lambda_{\alpha}^{(i)} - \nu_i\big) e_{\alpha}^{(i)} \wedge I e_{\alpha}^{(i)}$.
As a consequence $ \sum_{i =1}^m \sum_{\alpha \in P_i} \big(\lambda_{\alpha}^{(i)} - \nu_i\big) e_{\alpha}^{(i)} \wedge I e_{\alpha}^{(i)}$ is ${\rm d} {\rm d}^{\rm c}$-closed.

Since $ \sum_{i =1}^m \sum_{\alpha \in P_i} \big(\lambda_{\alpha}^{(i)} - \nu_i\big) e_{\alpha}^{(i)} \wedge I e_{\alpha}^{(i)}$ is the pull-back by $\pi$ of a ${\rm d} {\rm d}^{\rm c}$-closed $(1,1)$-form on~$G/T$ and the Aepply cohomology of $G/T$ is isomorphic to the Dolbeault cohomology, we have that $ \sum_{i =1}^m \sum_{\alpha \in P_i} \big(\lambda_{\alpha}^{(i)} - \nu_i\big) e_{\alpha}^{(i)} \wedge I e_{\alpha}^{(i)}$ is $d$-closed, and so $ \sum_{i =1}^m \sum_{\alpha \in P_i} \big(\lambda_{\alpha}^{(i)} - \nu_i\big) e_{\alpha}^{(i)} \wedge I e_{\alpha}^{(i)}$ is a~linear combination of $ \omega_i = {\rm d} \alpha_i$.
\end{proof}

\begin{Example} We will now construct an explicit $5$-parameter family of SKT metrics on the real compact Lie group ${\rm SO}(9)$ compatible with a Samelson complex structure. Recall that the complexification of ${\rm SO}(9)$ has rank 4 and its Lie algebra corresponds to $B_4$. Denote by $E_{ij}$ the matrix in ${\mathfrak {so}}(9, \mathbb C)$, whose unique non-zero entry is 1 at the place $(i,j)$. Let fix the maximal abelian subalgebra $\mathfrak t$ of ${\mathfrak {so}} (9)$ whose complexification $\mathfrak h= \mathfrak t ^{\rm c} = \operatorname{span} \langle H_1, H_2, H_3, H_4\rangle$, where
\begin{gather*}
H_1 = {\rm i} (E_{12} - E_{21}), \qquad\! H_2 = {\rm i} (E_{34} - E_{43}), \qquad\! H_3 = {\rm i} (E_{56} - E_{65}), \qquad\! H_4 = {\rm i} (E_{78} - E_{87}).\!
\end{gather*}
The weights for the fundamental representation of the Cartan subalgebra are $\pm f^k$, where
\[
f^1 = (1,0,0,0), \quad\dots,\quad f^4 = (0,0,0,1),
\]
 and the roots are given by $\pm f^j \pm f^k$, for $k \neq j$ and $\pm f^j$. In particular, the set $P$ of positive roots is the set consisting of $ f^j \pm f^k ,$ for $ j < k,$ and $f^j$ (see, for instance, \cite{St} for more details). We can choose a basis $(E_{\pm \alpha}, H_{\alpha})_{\alpha \in P}$ such that
 \[
 \mathfrak g_{\pm \alpha}= {\rm {span}}_{\mathbb C} \langle E_{\pm\alpha}\rangle, \qquad [E_{\alpha}, E_{-\alpha}]=H_{\alpha}, \qquad B(E_{\alpha}, E_{-\alpha}) = 1, \qquad [H, E_{\alpha}] = \alpha(H)E_{\alpha}
 \]
 and $B(H, H_{\alpha}) = \alpha(H)$ for all $H \in \mathfrak t$. The vectors $(E_{ \pm \alpha})$ satisfying the previous conditions are given by
\begin{gather*}
E_1^+ = \frac{\sqrt{2}}{2} (E_{19} + {\rm i} E_{29} - E_{91} - {\rm i} E_{92}), \qquad E_1^- = \frac{\sqrt{2}}{2} (E_{19} - {\rm i} E_{29} - E_{91}+ {\rm i} E_{92}),\\
E_2^+ = \frac{\sqrt{2}}{2} (E_{39} + {\rm i} E_{49} - E_{93} - {\rm i} E_{94}), \qquad E_2^- = \frac{\sqrt{2}}{2} (E_{39} - {\rm i} E_{49} - E_{93} + {\rm i} E_{94}),\\
E_3^+ = \frac{\sqrt{2}}{2} (E_{59} + {\rm i} E_{69} - E_{95} - {\rm i} E_{96}), \qquad E_3^- = \frac{\sqrt{2}}{2} (E_{59} - {\rm i} E_{69} - E_{95} + {\rm i} E_{96}),\\
E_4^+ = \frac{\sqrt{2}}{2} (E_{79} + {\rm i} E_{89} - E_{97} - {\rm i} E_{98}), \qquad E_4^- = \frac{\sqrt{2}}{2} (E_{79} - {\rm i} E_{89} - E_{97} + {\rm i} E_{98}),\\
[E_1^+, E_2^+], \qquad [E_1^+, E_3^+], \qquad [E_1^+, E_4^+], \qquad [E_2^+, E_3^+], \qquad [E_2^+, E_4^+], \qquad [E_3^+, E_4^+],\\
[E_1^-, E_2^-], \qquad [E_1^-, E_3^-], \qquad [E_1^-, E_4^-], \qquad [E_2^-, E_3^-], \qquad [E_2^-, E_4^-], \qquad [E_3^-, E_4^-],\\
[E_1^+, E_2^-], \qquad [E_1^+, E_3^-], \qquad [E_1^+, E_4^-], \qquad [E_2^+, E_1^-], \qquad [E_2^+, E_3^-], \qquad [E_2^+, E_4^-], \\
 [E_3^+, E_1^-], \qquad [E_3^+, E_2^-], \qquad [E_3^+, E_4^-], \qquad [E_4^+, E_1^-], \qquad [E_4^+, E_2^-], \qquad [E_4^+, E_3^-].
\end{gather*}
Therefore, we have
$\mathfrak t = \operatorname{span}\langle {\rm i}H_1, {\rm i}H_2, {\rm i}H_3, {\rm i}H_4\rangle$ and ${\mathfrak{so}}(9)$ is the real span of $(X_k, Y_k, {\rm i}H_j)$, with $k = 1, \dots, 16$, $j = 1, \dots, 4$,
and
\begin{gather*}
 {\rm i}H_1 = - (E_{12} - E_{21}), \qquad {\rm i}H_2 = -(E_{34} - E_{43}),\\
 {\rm i}H_3 = - (E_{56} - E_{65}), \qquad {\rm i}H_4 =- (E_{78} - E_{87}),\\
 X_1 = E_{19} - E_{91}, \qquad Y_1 = E_{29} - E_{92}, \qquad X_2 = E_{39} - E_{93}, \qquad Y_2 = E_{49} - E_{94},\\
 X_3 = E_{59} - E_{95}, \qquad Y_3 = E_{69} - E_{96}, \qquad X_4 = E_{79} - E_{97}, \qquad Y_4 = E_{89} - E_{98},\\
 X_5 = \frac{\sqrt{2}}{2} (-E_{13} + E_{24} + E_{31} - E_{42}), \qquad Y_5 = \frac{\sqrt{2}}{2} (-E_{14} - E_{23} + E_{41} + E_{32}),\\
 X_6 = \frac{\sqrt{2}}{2} (-E_{15} + E_{26} + E_{51} - E_{62}), \qquad Y_6 = \frac{\sqrt{2}}{2} (-E_{16} - E_{25} + E_{61} + E_{52}),\\
 X_7 = \frac{\sqrt{2}}{2} (-E_{17} + E_{28} + E_{71} - E_{82}), \qquad Y_7 = \frac{\sqrt{2}}{2} (-E_{18} - E_{27} + E_{81} + E_{72}),\\
 X_8 = \frac{\sqrt{2}}{2} (-E_{35} + E_{46} + E_{53} - E_{64}), \qquad Y_8 = \frac{\sqrt{2}}{2} (-E_{36} - E_{45} + E_{63} + E_{54}),\\
 X_9 = \frac{\sqrt{2}}{2} (-E_{37} + E_{48} + E_{73} - E_{84}), \qquad Y_9 = \frac{\sqrt{2}}{2} (-E_{38} - E_{47} + E_{83} + E_{74}),\\
 X_{10} = \frac{\sqrt{2}}{2} (-E_{57} + E_{68} + E_{75} - E_{86}), \qquad Y_{10} = \frac{\sqrt{2}}{2} (-E_{58} - E_{67} + E_{85} + E_{76}),\\
 X_{11} = \frac{\sqrt{2}}{2} (-E_{13} - E_{24} + E_{31} + E_{42}), \qquad Y_{11} = \frac{\sqrt{2}}{2} (E_{14} - E_{23} - E_{41} + E_{32}),\\
 X_{12} = \frac{\sqrt{2}}{2} (-E_{15} - E_{26} + E_{51} + E_{62}), \qquad Y_{12} = \frac{\sqrt{2}}{2} (E_{16} - E_{25} - E_{61} + E_{52}),\\
 X_{13} = \frac{\sqrt{2}}{2} (-E_{17} - E_{28} + E_{71} + E_{82}), \qquad Y_{13} = \frac{\sqrt{2}}{2} (E_{18} - E_{27} - E_{81} + E_{72}),
\\
 X_{14} = \frac{\sqrt{2}}{2} (-E_{35} - E_{46} + E_{53} + E_{64}), \qquad Y_{14} = \frac{\sqrt{2}}{2} (E_{36} - E_{45} - E_{63} + E_{54}),\\
 X_{15} = \frac{\sqrt{2}}{2} (-E_{37} - E_{48} + E_{73} + E_{84}), \qquad Y_{15} = \frac{\sqrt{2}}{2} (E_{38} - E_{47} - E_{83} + E_{74}),\\
 X_{16} = \frac{\sqrt{2}}{2} (-E_{57} - E_{68} + E_{75} + E_{86}), \qquad Y_{16} = \frac{\sqrt{2}}{2} (E_{58} - E_{67} - E_{85} + E_{76}).
\end{gather*}
If we denote by $I$ the Samelson complex structure defined by $I H_1= H_2$, $I H_3 = H_4$ and by the choice of positive roots, i.e., by $I X_j= Y_j$, $j = 1, \dots, 16$, we get that $({\rm SO}(9), I)$ has the following complex structure equations
\begin{gather*}
{\rm d} \varphi_1 = \frac{{\rm i}}{2} ( \varphi_3 \wedge \overline \varphi_3 + \varphi_7 \wedge \overline \varphi_7 + \varphi_8 \wedge \overline \varphi_8 + \varphi_9 \wedge \overline \varphi_9 + \varphi_{13} \wedge \overline \varphi_{13} + \varphi_{14} \wedge \overline \varphi_{14} +\varphi_{15} \wedge \overline \varphi_{15})\\
 \hphantom{{\rm d} \varphi_1 =}{} - \frac{1}{2} ( \varphi_4 \wedge \overline \varphi_4 + \varphi_7 \wedge \overline \varphi_7 + \varphi_{10} \wedge \overline \varphi_{10} + \varphi_{11} \wedge \overline \varphi_{11} - \varphi_{13} \wedge \overline \varphi_{13} +\varphi_{16} \wedge \overline \varphi_{16}\\
\hphantom{{\rm d} \varphi_1 =\frac{{\rm i}}{2} ( }{}
 +\varphi_{17} \wedge \overline \varphi_{17}),\\
{\rm d} \varphi_2 = \frac{{\rm i}}{2} ( \varphi_5 \wedge \overline \varphi_5 + \varphi_8 \wedge \overline \varphi_8 + \varphi_{10} \wedge \overline \varphi_{10} + \varphi_{12} \wedge \overline \varphi_{12} - \varphi_{14} \wedge \overline \varphi_{14} - \varphi_{16} \wedge \overline \varphi_{16} +\varphi_{18} \wedge \overline \varphi_{18})\\
\hphantom{{\rm d} \varphi_2 =}{}
 - \frac{1}{2} ( \varphi_6 \wedge \overline \varphi_6 + \varphi_9 \wedge \overline \varphi_9 + \varphi_{11} \wedge \overline \varphi_{11} + \varphi_{12} \wedge \overline \varphi_{12} - \varphi_{15} \wedge \overline \varphi_{15} -\varphi_{17} \wedge \overline \varphi_{17}\\
\hphantom{{\rm d} \varphi_2 =\frac{{\rm i}}{2} (}{}
 -\varphi_{18} \wedge \overline \varphi_{18}),\\
{\rm d} \varphi_3 = \frac{{\rm i}}{2} (\varphi_1 \wedge \varphi_3+ \overline \varphi_1 \wedge \varphi_3) - \frac{\sqrt{2}}{2} ( \overline \varphi_4 \wedge \varphi_7 + \varphi_4 \wedge \varphi_{13} + \overline \varphi_5 \wedge \varphi_8 + \varphi_5 \wedge \varphi_{14} )\\
\hphantom{{\rm d} \varphi_{9} =}{} - \frac{\sqrt{2}}{2} (\overline \varphi_6 \wedge \varphi_9 + \varphi_6 \wedge \varphi_{15}),\\
{\rm d} \varphi_4 = \frac{1}{2} (\varphi_1 \wedge \varphi_4 - \overline \varphi_1 \wedge \varphi_4)+ \frac{\sqrt{2}}{2} (\overline \varphi_3 \wedge \varphi_7 + \varphi_3 \wedge \overline \varphi_{13} - \overline \varphi_5 \wedge \varphi_{10} - \varphi_5 \wedge \varphi_{16} )\\
\hphantom{{\rm d} \varphi_{9} =}{} - \frac{\sqrt{2}}{2} (\varphi_6 \wedge \varphi_{17} + \overline \varphi_6 \wedge \varphi_{11})\\
{\rm d} \varphi_5 = \frac{{\rm i}}{2} (\varphi_2 \wedge \varphi_5+ \overline \varphi_2 \wedge \varphi_5) + \frac{\sqrt{2}}{2} ( \overline \varphi_3 \wedge \varphi_8 + \varphi_3 \wedge \overline \varphi_{14} + \overline \varphi_4 \wedge \varphi_{10} + \varphi_4 \wedge \overline \varphi_{16} )\\
\hphantom{{\rm d} \varphi_{9} =}{} - \frac{\sqrt{2}}{2} (\overline \varphi_6 \wedge \varphi_{12} + \varphi_6 \wedge \varphi_{18}),\\
{\rm d} \varphi_6 = \frac{1}{2} (\varphi_2 \wedge \varphi_6- \overline \varphi_2 \wedge \varphi_6) + \frac{\sqrt{2}}{2} ( \overline \varphi_3 \wedge \varphi_9 + \varphi_3 \wedge \overline \varphi_{15} + \overline \varphi_4 \wedge \varphi_{11} + \varphi_4 \wedge \overline \varphi_{17} )\\
\hphantom{{\rm d} \varphi_{9} =}{} + \frac{\sqrt{2}}{2} (\overline \varphi_5 \wedge \varphi_{12} + \varphi_5 \wedge \overline \varphi_{18}),\\
{\rm d} \varphi_7 = \frac{1}{2} (\varphi_1 \wedge \varphi_7 - \overline \varphi_1 \wedge \varphi_7) + \frac{{\rm i}}{2} ( \varphi_1 \wedge \varphi_7 + \overline \varphi_1 \wedge \varphi_7)\\
\hphantom{{\rm d} \varphi_{9} =}{} - \frac{\sqrt{2}}{2} (\varphi_3 \wedge \varphi_4 + \varphi_8 \wedge \varphi_{16} + \varphi_9 \wedge \varphi_{17} ) + \frac{\sqrt{2}}{2} (\varphi_{10} \wedge \varphi_{14} + \varphi_{11} \wedge \varphi_{15}),\\
{\rm d} \varphi_8 = \frac{{\rm i}}{2} (\varphi_1 \wedge \varphi_8 + \overline \varphi_1 \wedge \varphi_8) + \frac{{\rm i}}{2} (\varphi_2 \wedge \varphi_8 + \overline \varphi_2 \wedge \varphi_8)\\
\hphantom{{\rm d} \varphi_{9} =}{} - \frac{\sqrt{2}}{2} (\varphi_3 \wedge \varphi_5 - \varphi_7 \wedge \overline \varphi_{16} + \varphi_9 \wedge \varphi_{18} ) - \frac{\sqrt{2}}{2} (\varphi_{10} \wedge \varphi_{13} - \varphi_{12} \wedge \varphi_{15}),\\
{\rm d} \varphi_9 = \frac{{\rm i}}{2} (\varphi_1 \wedge \varphi_9 + \overline \varphi_1 \wedge \varphi_9) + \frac{1}{2} (\varphi_2 \wedge \varphi_9 - \overline \varphi_2 \wedge \varphi_9)\\
\hphantom{{\rm d} \varphi_{9} =}{} - \frac{\sqrt{2}}{2} (\varphi_3 \wedge \varphi_6 - \varphi_7 \wedge \overline \varphi_{17} - \varphi_8 \wedge \overline \varphi_{18} ) - \frac{\sqrt{2}}{2} (\varphi_{11} \wedge \varphi_{13} +\varphi_{12} \wedge \varphi_{14}),\\
{\rm d} \varphi_{10} = \frac{1}{2} (\varphi_1 \wedge \varphi_{10} - \overline \varphi_1 \wedge \varphi_{10}) + \frac{{\rm i}}{2} (\varphi_2 \wedge \varphi_{10} + \overline \varphi_2 \wedge \varphi_{10})\\
\hphantom{{\rm d} \varphi_{18} =}{} - \frac{\sqrt{2}}{2} (\varphi_4 \wedge \varphi_5 + \varphi_7 \wedge \overline \varphi_{14} - \varphi_8 \wedge \overline \varphi_{13} ) - \frac{\sqrt{2}}{2} (\varphi_{11} \wedge \varphi_{18} - \varphi_{12} \wedge \varphi_{17}),\\
{\rm d} \varphi_{11} = \frac{1}{2} (\varphi_1 \wedge \varphi_{11} - \overline \varphi_1 \wedge \varphi_{11}) + \frac{1}{2} (\varphi_2 \wedge \varphi_{11} - \overline \varphi_2 \wedge \varphi_{11}) \\
\hphantom{{\rm d} \varphi_{18} =}{} - \frac{\sqrt{2}}{2} (\varphi_4 \wedge \varphi_6 + \varphi_7 \wedge \overline \varphi_{15} - \varphi_9 \wedge \varphi_{13} ) + \frac{\sqrt{2}}{2} (\varphi_{10} \wedge \overline \varphi_{18} - \varphi_{12} \wedge \varphi_{16}),\\
{\rm d} \varphi_{12} = \frac{{\rm i}}{2} (\varphi_2 \wedge \varphi_{12} + \overline \varphi_2 \wedge \varphi_{12}) + \frac{1}{2} (\varphi_2 \wedge \varphi_{12} - \overline \varphi_2 \wedge \varphi_{12}) \\
\hphantom{{\rm d} \varphi_{18} =}{} - \frac{\sqrt{2}}{2} (\varphi_5 \wedge \varphi_6 + \varphi_8 \wedge \overline \varphi_{15} - \varphi_9 \wedge \overline \varphi_{14} ) - \frac{\sqrt{2}}{2} (\varphi_{10} \wedge \overline \varphi_{17} - \varphi_{11} \wedge \overline \varphi_{16}),\\
{\rm d} \varphi_{13} = \frac{{\rm i}}{2} (\varphi_1\wedge \varphi_{13} + \overline \varphi_1 \wedge \varphi_{13}) - \frac{1}{2} (\varphi_1 \wedge \varphi_{13} - \overline \varphi_1 \wedge \varphi_{13})\\
\hphantom{{\rm d} \varphi_{18} =}{} - \frac{\sqrt{2}}{2} (\varphi_3 \wedge \overline \varphi_4 + \varphi_8 \wedge \overline \varphi_{10} + \varphi_9 \wedge \overline \varphi_{11} ) - \frac{\sqrt{2}}{2} (\varphi_{14} \wedge \overline \varphi_{16} + \varphi_{15} \wedge \overline \varphi_{17}),
\\
{\rm d} \varphi_{14} = \frac{{\rm i}}{2} (\varphi_1 \wedge \varphi_{14} + \overline \varphi_1 \wedge \varphi_{14})- \frac{{\rm i}}{2} (\varphi_2 \wedge \varphi_{14} +\overline \varphi_2 \wedge \varphi_{14}) \\
\hphantom{{\rm d} \varphi_{18} =}{} - \frac{\sqrt{2}}{2} (\varphi_3 \wedge \overline \varphi_5 - \varphi_7 \wedge \overline \varphi_{10} + \varphi_9 \wedge \overline \varphi_{12} ) +\frac{\sqrt{2}}{2} (\varphi_{13} \wedge \varphi_{16} - \varphi_{15} \wedge \overline \varphi_{18}),\\
{\rm d} \varphi_{15} = \frac{{\rm i}}{2} (\varphi_1 \wedge \varphi_{15} + \overline \varphi_1 \wedge \varphi_{15}) - \frac{1}{2} (\varphi_2 \wedge \varphi_{15} - \overline \varphi_2 \wedge \varphi_{15})\\
\hphantom{{\rm d} \varphi_{18} =}{} - \frac{\sqrt{2}}{2} (\varphi_3 \wedge \overline \varphi_6 - \varphi_7 \wedge \overline \varphi_{11} - \varphi_8 \wedge \overline \varphi_{12} ) + \frac{\sqrt{2}}{2} (\varphi_{13} \wedge \varphi_{17} + \varphi_{14} \wedge \varphi_{18}),\\
{\rm d} \varphi_{16} = \frac{1}{2} (\varphi_1 \wedge \varphi_{16} - \overline \varphi_1 \wedge \varphi_{16}) - \frac{{\rm i}}{2} (\varphi_2 \wedge \varphi_{16} + \overline \varphi_2 \wedge \varphi_{16})\\
\hphantom{{\rm d} \varphi_{18} =}{} - \frac{\sqrt{2}}{2} (\varphi_4 \wedge \overline \varphi_5 + \varphi_7 \wedge \overline \varphi_{8} + \varphi_{11} \wedge \overline \varphi_{12} ) - \frac{\sqrt{2}}{2} ( \overline \varphi_{13} \wedge \overline \varphi_{14} + \varphi_{17} \wedge \overline \varphi_{18}),\\
{\rm d} \varphi_{17} = \frac{1}{2} (\varphi_1 \wedge \varphi_{17} - \overline \varphi_1 \wedge \varphi_{17}) - \frac{1}{2} (\varphi_2 \wedge \varphi_{17} - \overline \varphi_2 \wedge \varphi_{17}) \\
\hphantom{{\rm d} \varphi_{18} =}{}- \frac{\sqrt{2}}{2} (\varphi_4 \wedge \overline \varphi_6 + \varphi_7 \wedge \overline \varphi_{9} - \varphi_{10} \wedge \overline \varphi_{12} ) - \frac{\sqrt{2}}{2} ( \overline \varphi_{13} \wedge \varphi_{15} - \varphi_{16} \wedge \varphi_{18}),\\
{\rm d} \varphi_{18} = \frac{{\rm i}}{2} (\varphi_2 \wedge \varphi_{18} + \overline \varphi_2 \wedge \varphi_{18}) - \frac{1}{2} (\varphi_2 \wedge \varphi_{18} - \overline \varphi_2 \wedge \varphi_{18}) \\
\hphantom{{\rm d} \varphi_{18} =}{}
- \frac{\sqrt{2}}{2} (\varphi_5 \wedge \overline \varphi_6 + \varphi_8 \wedge \overline \varphi_{9} +\varphi_{10} \wedge \overline \varphi_{11} )
 - \frac{\sqrt{2}}{2} ( \overline \varphi_{14} \wedge \varphi_{15} + \overline \varphi_{16} \wedge \varphi_{17}).
\end{gather*}
By a direct computation, we have that the left-invariant metric $F = \sum_{k =1}^{18} {\rm i}b_k \varphi_k \wedge \overline \varphi_k$, with $b_k > 0$ for every $k$, is SKT if and only if
\begin{gather*}
F = \frac{{\rm i}}{2} (b_{16} - b_{17} + b_{18}) (\varphi_1 \wedge \overline \varphi_1 + \varphi_2 \wedge \overline \varphi_2) + \frac{{\rm i}}{2} (b_6 +b_{15} - b_{16} + b_{17} - b_{18}) \varphi_3 \wedge \overline \varphi_3\\
\hphantom{F = }{} + \frac{{\rm i}}{2} (b_6 - b_{16} + 2 b_{17} - b_{18}) \varphi_4 \wedge \overline \varphi_4 + \frac{{\rm i}}{2} (b_6 + b_{17} - b_{16}) \varphi_5 \wedge \overline \varphi_5 + b_6 \varphi_6 \wedge \overline \varphi_6\\
\hphantom{F = }{} + \frac{{\rm i}}{2} (2 b_6 + b_{15} - 3 b_{16} + 4 b_{17} - 3 b_{18}) \varphi_7 \wedge \overline \varphi_7\!+ \frac{{\rm i}}{2} (2 b_6 + b_{15} -3 b_{16} - 2 b_{18} + 3 b_{17}) \varphi_8 \wedge \overline \varphi_8\!\\
\hphantom{F = }{} + \frac{{\rm i}}{2} ( 2 b_6 + b_{15} - 2 b_{16} + 2 b_{17} - 2 b_{18}) \varphi_9 \wedge \overline \varphi_9 + \frac{{\rm i}}{2} (2 b_6 - 3 b_{16} + 4 b_{17} - 2 b_{18}) \varphi_{10} \wedge \overline \varphi_{10}\\
\hphantom{F = }{} + \frac{{\rm i}}{2} (3 b_{17} - 2 b_{18} + 2 b_6 - 2 b_{16}) \varphi_{11} \wedge \overline \varphi_{11} + \frac{{\rm i}}{2} (2 b_6 - 2 b_{16} + 2 b_{17} - b_{18}) \varphi_{12} \wedge \overline \varphi_{12}\\
\hphantom{F = }{} + \frac{{\rm i}}{2} ( b_{15} + b_{16} -2 b_{17} +b_{18}) \varphi_{13} \wedge \overline \varphi_{13} + \frac{{\rm i}}{2} (b_{15} + b_{16} - b_{17}) \varphi_{14} \wedge \overline \varphi_{14}\\
\hphantom{F = }{} + b_{15} \varphi_{15} \wedge \overline \varphi_{15}+ b_{16} \varphi_{16} \wedge \overline \varphi_{16}+ b_{17} \varphi_{17} \wedge \overline \varphi_{17}+ b_{18} \varphi_{18} \wedge \overline \varphi_{18}.
\end{gather*}
So we obtain a $5$-parameter family of left-invariant $I$-Hermitian SKT metrics which are also right $T$-invariant. By the results in the next section (Corollary \ref{coroBismutflat}) this family describes all possible such metrics.
Moreover, $F = F_0$, where $F_0$ is the fundamental form associated to $- B$, if and only if $b_{15}= b_{16}= b_{17}= b_{18} = b_6 =1$. Since
\begin{gather*}
F = (b_{16} - b_{17} + b_{18}) F_0 + \frac{1}{2} (b_6 + b_{15} -2 b_{16} + 2 b_{17} -2 b_{18}) {\rm d} (\varphi_1 + \overline \varphi_1)\\[2pt]
\hphantom{F=}{} + \frac{{\rm i}}{2} (b_6 - 2b_{16} + 3 b_{17} - 2 b_{18}) {\rm d} (\overline \varphi_1 - \varphi_1)
+ \frac 12 (b_6 -2 b_{16} + 2 b_{17} - b_{18}]) {\rm d} ( \varphi_2 +\overline \varphi_2)\\[2pt]
\hphantom{F=}{} + \frac{1}{2} (b_6 - b_{16} + b_{17} - b_{18}) {\rm d} (\overline \varphi_2 - \varphi_2),
\end{gather*}
 the Aeppli class of $F$ coincides with the one of $F_0$. Note that by doing similar computations as in \cite{Pittie,Pittie2} (see also~\cite{LUV}) it is possible to describe the minimal model for the Dolbeault cohomology of $({\rm SO}(9), I)$. Calculation for the Aeppli cohomology of the compact simple Lie groups of rank two is given in~\cite{Barbaro}.
\end{Example}

\section{ Invariant CYT metrics}

We first mention an analog of Lemma \ref{llemma1} for CYT metrics on compact semi-simple Lie groups:

\begin{Proposition}
Let $I$ be a Samelson's complex structure on a compact semi-simple Lie group~$G$ and $g$ be a left-invariant CYT $I$-Hermitian metric. Then there is a left-invariant CYT $I$-Hermitian metric on $(G,I)$ which is also right $T$-invariant.
\end{Proposition}

\begin{proof}
We use the same idea as in \cite{FG} combined with the symmetrization of Lemma~\ref{llemma1}.
By Koszul~\cite{K},
the Ricci form of the Chern connection for every left-invariant $I$-Hermitian metric on $(G, I)$, is the exterior derivative of the half-sum of all positive roots, i.e.,
\[
\rho^{\rm Ch}= \frac{1}{2} {\rm i} \sum_{\alpha \in P} {\rm d} \alpha.
\]
Since the Bismut Ricci form $\rho^{\rm B}$ is related to the Chern Ricci form by the relation $\rho^{\rm B} = \rho^{\rm Ch} + {\rm d} (\delta F)$, where $F$ denotes the fundamental form associated to the Hermitian metric, to prove the proposition it is enough to find a right $T$-invariant $I$-Hermitian metric $h$ whose fundamental form~$F$ satisfies
\[
\delta F = - \frac{1}{2} {\rm i} \sum_{\alpha \in P} \alpha = \sigma,
\]
where $\delta = - \star_h d \star_h$ denotes the codifferential with respect to $h$ and $\sigma$ is a left-invariant $1$-form. This follows from the fact that ${\rm d} (\delta F - \sigma)=0$ leads to $\delta F = \sigma$ because on a compact semi-simple Lie group a closed left-invariant $1$-form has to be trivial.
Since $ \star_h F = - \frac{1}{(n -1)!} F^{n-1}$, where $2n$ is the real dimension of $G$, the previous condition is equivalent to ${\rm d} F^{n-1} = (n -1)! \star_h (\sigma) =- I\sigma\wedge F^{n-1}$.
Now if we symmetrize $F^{n-1}$
with respect to the right $T$-action as in Lemma~\ref{llemma1} we obtain a right $T$-invariant $(n-1,n-1)$-form which is also positive, so it is of the type $\tilde{F}^{n-1}$ for a right $T$-invariant form $\tilde{F}$. Since $\sigma$ and $I \sigma$ belong to $\mathfrak t^*$ and the ${\rm Ad}(T)$-action on $\mathfrak t$ is trivial, then ${\rm d} \tilde{F}^{n-1} = -I\sigma \wedge \tilde{F}^{n-1}$.
Since $F$ and $\tilde{F}$ coincide on $\mathfrak t$, the proof is completed.
\end{proof}

Let $g$ be a left-invariant $I$-Hermitian metric on $G$ which is also invariant by the right $T$-action. Then we can write $g (\cdot, \cdot ) = -B (\Lambda (\cdot), \cdot)$, where
$\Lambda$ is a positive-definite Hermitian matrix on $\mathfrak t$ and diagonal on $\sum_{\alpha} \mathfrak g_{\alpha}$.
In particular, $g (E_{\alpha}, E_{-\alpha}) = \lambda_{\alpha}$ and $g(E_{\alpha},E_{\beta}) =0$, for every $\alpha \neq - \beta$.

Let $F(\cdot, \cdot) = g (I \cdot, \cdot)$ be the fundamental form associated to the Hermitian structure $(I, g)$.
We can calculate the codifferential $\delta F$ using the formula in \cite{Gau}:
\begin{equation*}
I \circ \delta F (X) = \frac 12 \sum_{l=1}^{n} {\rm d}F (e_l, I e_l, X),
\end{equation*}
where $n$ is the complex dimension of $G$ and $(e_i)$ is a $I$-adapted orthonormal basis of $(G, I, g)$, such that $I e_{2l -1} = e_{2l}$, $l = 1, \dots, n$. We can choose as $I$-adapted orthonormal basis $(e_i)$ of~$(G, I, g)$
\[
\left(\frac{X_{\alpha}}{ \sqrt{-\lambda_{\alpha}}},\frac{Y_{\alpha}}{ \sqrt{- \lambda_{\alpha}}}, H_k \right),
\]
where $(H_k)$ is an orthonormal basis of $\mathfrak t$ such that $I H_{2l -1} = H_{2l}$.
Because
\[
{\rm d}F(X,Y,Z)=F([X,Y],Z)+F([Y,Z],X)+F([Z,X],Y)
\]
 for every left-invariant $2$-form $F$ on $G$ and every $X,Y, Z \in \mathfrak g$, we get that the sum
 \[
 \sum_{l=1}^{n} {\rm d}F(e_l, Je_l, X) = \sum_{\alpha \in P} \frac{1}{\lambda_{\alpha}} {\rm d}F(X_{\alpha}, Y_{\alpha}, X) + \sum_{l =1}^{r/2} {\rm d}F( H_{2l -1}, H_{2l}, X)
 \]
 has
contributions only from
\begin{align*}
{\rm d}{\rm d}F(X_{\alpha}, Y_{\alpha}, Z) &= {\rm i}{\rm d}F{\rm d}(E_{\alpha},E_{-\alpha}, Z) = {\rm i}{\rm d}([E_{\alpha}, E_{-\alpha}], Z)\\
& = {\rm i}g (I ([E_{\alpha}, E_{-\alpha}]), Z) = - {\rm i}g( H_{\alpha}, IZ)\\
& = {\rm i}B (\Lambda( H_{\alpha}), IZ) = {\rm i}B (H_{\alpha}, \Lambda IZ) = {\rm i}\alpha (\Lambda I Z)
\end{align*}
 for every $Z\in \mathfrak{h}$.

Then we get
\[
\delta F \circ
I = - \frac {1} {2} {\rm i} \sum_{\alpha \in P} \frac{1} {\lambda_{\alpha}} (\alpha \circ \Lambda) \circ I.
\]
 By the above, the Ricci form $\rho^{\rm B}$ is given by
\[
\rho^{\rm B} = \frac{1}{2} {\rm i} \sum_{\alpha \in P} \bigg({\rm d}\alpha - \frac{1}{\lambda_{\alpha}} {\rm d} (\alpha\circ\Lambda) \bigg).
\]

Therefore $g$ is CYT if and only $\sum_{\alpha \in P} \big({\rm d}\alpha - \frac{1}{\lambda_{\alpha}} {\rm d}(\alpha\circ\Lambda) \big ) =0$. Using this characterization we have:

\begin{Theorem} \label{th4.1}
Let $G$ be a compact semi-simple Lie group with a Samelson's complex structure determined by the set $P$ of positive roots associated to a maximal torus $T$ of $G$. Let $g$ be a~left-invariant $I$-Hermitian metric on $G$, which is also invariant under the right action of the torus $T$ in the Tits fibration $G\rightarrow G/T$, where $T$ is a maximal torus in $G$. We can write the fundamental form $F$ of the metric $g$
as $F = F|_{\mathfrak t} + \sum_{i=1}^m \sum_{\alpha \in P_i} \lambda_{\alpha} ^{(i)} e_{\alpha}\wedge I e_{\alpha}^{(i)}$, where $e_{\alpha}^{(i)}$, $Ie_{\alpha}^{(i)}$ are dual to~$X_{\alpha}$,~$Y_{\alpha}$ and $\lambda_{\alpha} ^{(i)}= g (E_{\alpha}, E_{-\alpha})$. Then
$g$ is CYT iff $\sum_{\alpha \in P} \big({\rm d}\alpha - \frac{1}{\lambda_{\alpha}} {\rm d}(\alpha\circ\Lambda) \big) = 0$, where $\Lambda$ is the positive operator such that $g |_{\mathfrak t} (\cdot, \cdot ) = -B (\Lambda (\cdot), \cdot)$.
\end{Theorem}

\section{CYT and SKT invariant metrics}

As a consequence of Corollary \ref{th4.1bis} and Theorem \ref{th4.1}, we can get the explicit equations to determine all left-invariant and right $T$-invariant CYT (or SKT) metrics on $(G, I).$ In particular, we can prove that if $g$ is both SKT and CYT, then $g$ has to be Bismut flat.

\begin{Corollary} \label{coroBismutflat}
Let $G$ be a compact semi-simple Lie group with Samelson's complex structure~$I$ determined by the set $P$ of positive roots associated to a maximal torus $T$ of $G$. Let $g$ be a~left-invariant $I$-Hermitian metric on $G$, which is also invariant under the right action of the torus~$T$ in the Tits fibration $G\rightarrow G/T$. Let $\alpha_1, \dots,\alpha_r$ be the choice of simple roots and $p$ be the cardinality of $P$.
Then
\begin{enumerate}
 \item[$(i)$] $g$ is CYT if and only if for every $\alpha \in P$ there exist $p$ positive real numbers $\lambda_{\alpha} > 0$, $\alpha \in P$, satisfying the $r$ equations
\begin{equation}\label{CYTsystem}
\sum_{\alpha \in P} \Bigg( x_\alpha^i - \sum_{j=1}^r \frac{x_\alpha^j}{\lambda_{\alpha}} \Lambda^i_j \Bigg) =0, \qquad i = 1, \dots, r,
\end{equation}
where $x_{\alpha}^j$ denotes the coefficient of $\alpha$ with respect to the simple root $\alpha_j$ and $\Lambda_i^j$ are the entries of $\Lambda$ in the basis $(\alpha_1, \dots,\alpha_r)$.
 \item[$(ii)$] $g$ is SKT if and only if there exist $r$ real numbers $ \tilde N_j$ and $p$ positive real numbers $ \tilde \lambda_{\alpha} >0$, $\alpha \in P$, satisfying the $p$ equations
\begin{equation}\label{SKTsystem}
\tilde \lambda_{\alpha} - 1 - \sum_{j = 1}^r \tilde N_j B(H_{\alpha_j}, H_{\alpha})=0, \qquad \alpha \in P.
\end{equation}
 \item[$(iii)$] If $g$ is both SKT and CYT, then $g$ is Bismut flat.
 \end{enumerate}
\end{Corollary}

\begin{proof}
From Theorem \ref{th4.1}, it follows that $g$ is CYT if and only if there exist $p$ positive real number $\lambda_{\alpha}$, $\alpha \in P,$ such that
\[
\sum_{\alpha \in P}  \bigg ({\rm d}\alpha - \frac{1}{\lambda_{\alpha} } {\rm d}( \alpha \circ \Lambda ) \bigg) =0.
\]
Therefore, since $\alpha =\sum_{i = 1}^r x_\alpha^i \alpha_i$, with $x_\alpha^i$ positive integers and $\Lambda_i^j$ the entries of $\Lambda$,
 we have that the CYT condition is equivalent to
\[
\sum_{\alpha \in P} \Bigg( x_\alpha^i - \sum_{j=1}^r \frac{x_\alpha^j}{\lambda_{\alpha}} \Lambda^i_j \Bigg) {\rm d}\alpha_i = 0.
\]
Since the forms ${\rm d} \alpha_i$, $i= 1, \dots, r$, are linearly independent, we get the non-linear system \eqref{CYTsystem} of~$r$ equations in the $p$ variables $\lambda_{\alpha}$, $\alpha \in P$, and (i)~follows.

Let denote by $P_i$ the set of positive roots of the simple factor $G_i$ with respect to $\mathfrak t_i$.
To prove the second part of the corollary we use that by Corollary \ref{th4.1bis} the SKT condition is equivalent to the existence of $r$ real numbers $N_i$, $i = 1, \dots, r$ and $p$ positive numbers $\lambda_{\alpha}$, $\alpha \in P$, such that
\[
\sum_{i = 1}^m \sum_{\alpha \in P_i} \big(\lambda_{\alpha}^{(i)} - \nu_i\big) e_{\alpha}^{(i)} \wedge I e_{\alpha}^{(i)} = \sum_{i = 1}^r N_i \, {\rm d}\alpha_i
\]
and $\Lambda = \sum_{i =1}^m \nu_i {\rm Id} |_{\mathfrak t_i}$.
Since
\[
{\rm d}\alpha_j= \sum_{i = 1}^m \sum_{\alpha \in P_i} B\big(H_{\alpha_j}, H_{\alpha}^{(i)}\big) e_{\alpha} ^{(i)} \wedge I e_{\alpha}^{(i)}, \qquad j =1, \dots, r,
\]
we obtain
\[
\sum_{i =1}^m \sum_{\alpha \in P_i} \big(\lambda_{\alpha}^{(i)} - \nu_i\big) e_{\alpha}^{(i)} \wedge I e_{\alpha} ^{(i)}= \sum_{i =1}^m \sum_{j = 1}^r N_j \sum_{\alpha \in P_i} B\big(H_{\alpha_j}, H_{\alpha}^{(i)}\big) e_{\alpha}^{(i)} \wedge I e_{\alpha}^{(i)}.
\]
Therefore, using that the previous equality holds on every simple factor $\mathfrak g_i$ and rescaling $\lambda_{\alpha}^{(i)}$ and $N_i$ by $\nu_i$ we get the linear system \eqref{SKTsystem} of $p$ equations in the $p + r$ variables $\tilde N_i$, $i = 1, \dots, r,$ and $ \tilde \lambda_{\alpha}$, $\alpha \in P.$

Now to prove the last part of the theorem suppose that $g$ is both CYT and SKT. Then using that $\Lambda = \sum_{i =1}^m \nu_i {\rm Id} |_{\mathfrak t_i}$ and rescaling $\lambda_{\alpha}$ in \eqref{CYTsystem} in the same way as for the SKT case we have that the equations for the CYT condition become
\[
\sum_{\alpha \in P} x_\alpha^i\Bigg(1 - \sum_{j =1}^r \frac{1}{\tilde \lambda_{\alpha}} \Bigg) = 0, \qquad i = 1, \dots ,r.
\]
Let
\[
P = \{ \alpha_1, \dots, \alpha_r, \alpha_{r + 1}, \dots, \alpha_p \},
\]
where $\alpha_1, \dots, \alpha_r$ are the simple roots.
Imposing the SKT and CYT condition we get the following non-linear system of $p + r$ equations in the variables $\tilde \lambda_{j}$, $j =1, \dots, p$, $\tilde N_i$, $i = 1, \dots, r$,
\begin{equation}\label{newsystem}
 \begin{cases}
\displaystyle \tilde \lambda_{j}-1 - \sum_{i = 1}^r \tilde N_i B(H_{\alpha_i}, H_{\alpha_j})=0, \qquad j = 1, \dots, p,\\
\displaystyle \sum_{j=1}^p x_{\alpha_j}^i \bigg(1 - \frac{1}{\tilde \lambda_{j}} \bigg) =0, \qquad i = 1, \dots, r.
\end{cases}
\end{equation}
Note that the solution $\tilde N_i =0$, $i=1, \dots, r$, $\tilde \lambda_j = 1$, $j = 1, \dots, p$, corresponds to a Bismut flat metric.
Since $ \tilde \lambda_j \neq 0$, for every $j = 1, \dots, p$, if we multiply
\[
\sum_{j=1}^p x_{\alpha_j}^i \bigg(1 - \frac{1}{\tilde \lambda_{j}} \bigg) =0, \qquad i = 1, \dots, r,
\]
by the product of all the $\tilde \lambda_j$,we get
\[
\sum_{j=1}^p x_{\alpha_j}^i \tilde \lambda_1 \cdots \hat{\tilde {\lambda_j}} \cdots \tilde \lambda _p \big(\tilde \lambda_j - 1\big) =0, \qquad i = 1, \dots, r,
\]
where $\hat{ \tilde {\lambda_{j} }}$ denotes that the element $\tilde \lambda_{j}$ is removed in the product. We know that
\[
\tilde \lambda_{j}-1 = \sum_{k = 1}^r \tilde N_k B(H_{\alpha_k}, H_{\alpha_j})
\]
for every $j = 1, \dots, p$. Therefore
\[
\sum_{j=1}^p x_{\alpha_j}^i \tilde \lambda_1 \cdots \hat{\tilde {\lambda_j}} \cdots \tilde \lambda _p \Bigg (\sum_{k = 1}^r \tilde N_k B(H_{\alpha_k}, H_{\alpha_j})\Bigg)=0, \qquad i = 1, \dots, r,
\]
 which can be viewed as a homogeneous system of $r$ equations in the $r$ variables $\tilde N_i$, $i = 1, \dots, r$. The $r \times r$ associated matrix $M = M\big(\tilde \lambda_1, \dots, \tilde \lambda_p\big)$ has $(i,k)$-th entry given by
\[
\sum_{j=1}^p x_{\alpha_j}^i \tilde \lambda_1 \cdots \hat{\tilde {\lambda_j}} \cdots \tilde \lambda _p B(H_{\alpha_k}, H_{\alpha_j}).
\]
We can rewrite $B(H_{\alpha_k}, H_{\alpha_j})$, $k =1, \dots, r$, $j = 1, \dots, p,$ in terms of $B(H_{\alpha_k}, H_{\alpha_j})$ with $k, j \in \{1, \dots, r \}$, using
 that $B (H, H_{\alpha}) = \alpha (H)$ for every $H \in \mathfrak t$. Indeed, if $\alpha = \sum_{l=1}^r x_{\alpha}^l \alpha_l$ we have,
\[
B (H_{\alpha_i}, H_{\alpha}) = \alpha (H_{\alpha_i}) =
 \sum_{l=1}^r x_{\alpha}^l \alpha_l (H_{\alpha_i}) = \sum_{l=1}^r x_{\alpha}^l B(H_{\alpha_i}, H_{\alpha_l}),
\]
 for every $i = 1, \dots, r$.
 As consequence the $(i,k)$-th entry of the matrix $M = M\big(\tilde \lambda_1, \dots, \tilde \lambda_p\big)$ can be written as
\begin{gather*}
\sum_{j=1}^p x_{\alpha_j}^i \tilde \lambda_1 \cdots \hat{\tilde {\lambda_j}} \cdots \tilde \lambda _p B(H_{\alpha_k}, H_{\alpha_j}) = \sum_{j=1}^p x_{\alpha_j}^i \tilde \lambda_1 \cdots \hat{\tilde {\lambda_j}} \cdots \tilde \lambda _p B(H_{\alpha_j}, H_{\alpha_k})\\
\hphantom{\sum_{j=1}^p x_{\alpha_j}^i \tilde \lambda_1 \cdots \hat{\tilde {\lambda_j}} \cdots \tilde \lambda _p B(H_{\alpha_k}, H_{\alpha_j})}{} = \sum_{j=1}^p x_{\alpha_j}^i \tilde \lambda_1 \cdots \hat{\tilde {\lambda_j}} \cdots \tilde \lambda _p \sum_{l=1}^r x_{\alpha_j}^l B(H_{\alpha_l}, H_{\alpha_k})\\
 \hphantom{\sum_{j=1}^p x_{\alpha_j}^i \tilde \lambda_1 \cdots \hat{\tilde {\lambda_j}} \cdots \tilde \lambda _p B(H_{\alpha_k}, H_{\alpha_j})}{} = B\Bigg( \sum_{l=1}^r \sum_{j=1}^p x_{\alpha_j}^i \tilde \lambda_1 \cdots \hat{\tilde {\lambda_j}} \cdots \tilde \lambda _p x_{\alpha_j}^l H_{\alpha_l}, H_{\alpha_k}\Bigg),
\end{gather*}
where $i , k = 1, \dots, r$. Since
\[
 \sum_{j=1}^p x_{\alpha_j}^i \tilde \lambda_1 \cdots \hat{\tilde {\lambda_j}} \cdots \tilde \lambda _p x_{\alpha_j}^l = \tilde \lambda_1 \cdots \tilde \lambda_p \sum_{j=1}^p x_{\alpha_j}^i \frac{1} { \tilde \lambda_j} x_{\alpha_j}^l,
\]
 it will be sufficient to show, for every $ \tilde \lambda_j$, the invertibility of the matrix $L$, whose $(i,l)$-entry is given
\[
 \sum_{j=1}^p x_{\alpha_j}^i \tilde \lambda_1 \cdots \hat{\tilde {\lambda_j}} \cdots \tilde \lambda _p x_{\alpha_j}^l = \tilde \lambda_1 \cdots \tilde \lambda_p \sum_{j=1}^p x_{\alpha_j}^i \frac{1} {\tilde \lambda_j} x_{\alpha_j}^l.
\]
Since the $r \times p$ matrix $(x_{\alpha}^i)$ has rank $r$, the $r$ vectors
\[
v_1 = \big(x^1_{\alpha_1}, \dots, x^1_{\alpha_p}\big), \quad \dots, \quad v_r = \big(x^r_{\alpha_1}, \dots, x^r_{\alpha_p}\big),
\]
are linearly independent. Using that for every $j=1, \dots, p,$ $ \tilde \lambda_j > 0$, we get that $\sum_{j=1}^p \frac{1} { \tilde \lambda_j} x_{\alpha_j}^i x_{\alpha_j}^l $ can be viewed as a positive definite scalar product on $\mathbb R^p$ with weights $\frac{1}{ \tilde \lambda_j}$ of the two vectors~$v_i$ and~$v_l$. Therefore,
the matrix $L$ is invertible and $\tilde N_i =0$, for every $i = 1, \dots, r$. By \eqref{newsystem}, it follows that $\tilde \lambda_j = 1$, for every $j =1, \dots, p$.
\end{proof}

\subsection*{Acknowledgements}

Anna Fino is partially supported by Project PRIN 2017 ``Real and complex manifolds: Topology, Geometry and Holomorphic Dynamics'', by GNSAGA (Indam) and by a grant from the Simons Foundation (\#944448). Gueo
Grantcharov is partially supported by a grant from the Simons
Foundation (\#853269). We would like to thank the anonymous referees for the helpful comments and suggestions as well as pointing out a confusing statement in the earlier version of the proof of Theorem~\ref{lemma2}.


\pdfbookmark[1]{References}{ref}
\LastPageEnding

\end{document}